\documentclass[a4paper,11pt]{amsart}

\usepackage{amsfonts,latexsym,rawfonts,amsmath,amssymb,amsthm, mathrsfs, lscape}
\usepackage[all]{xy}
\usepackage[english]{babel}
\usepackage[utf8]{inputenc}
\usepackage{xfrac}
\usepackage{multirow}

\usepackage{array, tabularx}

\usepackage{setspace}
\usepackage{textcomp}

\usepackage[hypertexnames=false,
backref=page,
    pdftex,
    pdfpagemode=UseNone,
    breaklinks=true,
    extension=pdf,
    colorlinks=true,
    linkcolor=blue,
    citecolor=blue,
    urlcolor=blue,
]{hyperref}

\usepackage[top=1.5in, bottom=1in, left=0.9in, right=0.9in]{geometry}

\newcommand{\ce}{\mathcal{C}^\infty}
\newcommand{\RR}{\mathbb{R}}

\newcommand{\referenza}{}

\newtheorem{thm}{Theorem}
\newtheorem*{thm*}{Theorem \referenza}
\newtheorem{cor}[thm]{Corollary}
\newtheorem*{cor*}{Corollary \referenza}

\newtheorem*{lem*}{Lemma \referenza}

\newtheorem{prop}[thm]{Proposition}
\newtheorem*{prop*}{Proposition \referenza}

\newtheorem*{conj*}{Conjecture \referenza}
\newtheorem{rmk}[thm]{Remark}

\newtheorem{defi}[thm]{Definition}

\def\XXint#1#2#3{{\setbox0=\hbox{$#1{#2#3}{\int}$ }
\vcenter{\hbox{$#2#3$ }}\kern-.6\wd0}}

\allowdisplaybreaks[1]

\title{Variational problems in conformal geometry}

\author{Daniele Angella}
\address[Daniele Angella]{
Dipartimento di Matematica e Informatica ``Ulisse Dini''\\
Universit\`a degli Studi di Firenze\\
viale Morgagni 67/a\\
50134 Firenze, Italy
}
\email{daniele.angella@gmail.com}
\email{daniele.angella@unifi.it}

\author{Nicolina Istrati}
\address[Nicolina Istrati]{School of Mathematical Sciences, Tel Aviv University, Ramat Aviv, Tel Aviv 69978 Israel}
\email{nicolinai@mail.tau.ac.il}

\author{Alexandra Otiman}
\address[Alexandra Otiman]{
Dipartimento di Matematica e Fisica\\
Universit\`a degli Studi Roma Tre\\
Via della Vasca Navale, 84\\
00146 Roma, Italy\\
and Institute of Mathematics ``Simion Stoilow'' of the Romanian Academy, 21,
Calea Grivitei Street, 010702, Bucharest, Romania
}
\email{aiotiman@mat.uniroma3.it}

\author{Nicoletta Tardini}
\address[Nicoletta Tardini]{Dipartimento di Matematica ``G. Peano'' \\
Universit\`{a} degli studi di Torino \\
Via Carlo Alberto 10\\
10123 Torino, Italy}
\email{nicoletta.tardini@gmail.com}
\email{nicoletta.tardini@unito.it}

\keywords{conformal Hermitian geometry; Gauduchon metric; Euler-Lagrange equation}
\thanks{The first-named author is supported by project SIR2014 ``Analytic aspects in complex and hypercomplex geometry'' (code RBSI14DYEB), by project PRIN 2017 ``Real and Complex Manifolds: Topology, Geometry and holomorphic dynamics'' (code 2017JZ2SW5), and by GNSAGA of INdAM.
The third-named author is partially supported by a grant of the Romanian Ministry of Research and Innovation, CNCS - UEFISCDI,
project number PN-III-P4-ID-PCE-2016-0065, within PNCDI III.
The fourth-named author is supported by project SIR2014 ``Analytic aspects in complex and hypercomplex geometry'' (code RBSI14DYEB) and by GNSAGA of INdAM}

\subjclass[2010]{53C55; 53A30; 58E11}

\dedicatory{
Dedicated to Max Pontecorvo for his 60th birthday.\\
Buon compleanno, Max!
}

\begin{document}

\begin{abstract}
We study the Euler-Lagrange equation for several natural functionals defined on a conformal class of almost Hermitian metrics, whose expression involves the Lee form $\theta$ of the metric. We show that the Gauduchon metrics are the unique extremal metrics of the functional corresponding to the norm of the codifferential of the Lee form.  We prove that on compact complex surfaces, in every conformal class there exists a unique metric, up to multiplication by a constant, which is extremal for the functional given by  the $L^2$-norm of $dJ\theta$, where $J$ denotes the complex structure. These extremal metrics are not  the Gauduchon metrics in general, hence we extend their definition to any dimension and show that they give unique representatives, up to constant multiples, of any conformal class of almost Hermitian metrics. 
\end{abstract}

\maketitle

\section*{Introduction}
Let $X$ be a compact smooth $2n$-dimensional manifold endowed with an almost complex structure $J$. Fix an almost Hermitian metric $g$ on $(X,J)$, which we will identify with its associated $(1,1)$-form $\Omega=g(J\_,\_)$. 
Consider the linear operator $L_\Omega:=\Omega\wedge\_$ and its adjoint $\Lambda_\Omega:=L_\Omega^*$. We recall that $L_\Omega^{n-1}\colon \wedge^1X \to \wedge^{2n-1}X$ is an isomorphism, therefore one can define the {\em torsion $1$-form} of $\Omega$, also known as the {\em Lee form}, as:
\begin{equation}\label{eq:theta-def}
\theta_\Omega := \Lambda_\Omega d\Omega = Jd^{*_\Omega}\Omega \in\wedge^1X
\end{equation}
such that
$$ d\Omega^{n-1}=\theta_\Omega\wedge\Omega^{n-1}. $$
Furthermore, the Lee form can be seen to coincide with the trace of the torsion of the Chern connection \cite[Th\'eor\`eme 3]{gauduchon-CRAS}. 

In general, the recipe for defining special classes of Hermitian metrics on complex manifolds is by imposing $\Omega$ to be in the kernel of a specific differential operator. A well-known list of examples include {\em strong K\" ahler with torsion} (SKT), also known as {\em pluriclosed}, ($dd^c\Omega=0$), {\em balanced} ($d\Omega^{n-1}=0$), {\em locally conformally K\" ahler} (lcK) ($d\Omega=\alpha\wedge \Omega$ for a closed one-form $\alpha$), {\em locally conformally balanced}  ($d\Omega^{n-1}=\theta \wedge \Omega^{n-1}$, for a closed one-form $\theta$), all of these being generalizations of the K\" ahler condition $d\Omega=0$.

Another natural class of Hermitian metrics is given by the \textit{Gauduchon metrics}, which are defined as having co-closed Lee form, or equivalently as satisfying $dd^c\Omega^{n-1}=0$. These metrics turned out to be a very useful tool in conformal geometry, since by the celebrated result \cite[Th\'eor\`eme 1]{gauduchon-CRAS}, any conformal class of  almost Hermitian metrics contains a Gauduchon metric, unique up to multiplication with positive constants, provided the manifold is compact and not a Riemann surface.

On the other hand, it is reasonable to expect that interesting classes of metrics arise as critical points of naturally defined functionals on the space of Hermitian metrics. In the general Hermitian setting, two such functionals have been considered in the literature. Let $\mathcal{H}_1$ be the set of almost Hermitian metrics on $X$ with total volume one.
\begin{itemize}
\item The {\em Gauduchon functional} \cite[Section III]{gauduchon} is defined by:
$$ \mathcal L_G \colon \mathcal{H}_1 \to \mathbb R, \qquad
\mathcal L_G(\Omega) := \int_X |\theta_\Omega|_\Omega^2 \frac{\Omega^n}{n!} = \frac{1}{(n-1)!} \int_X \theta_\Omega \wedge J\theta_\Omega \wedge \Omega^{n-1} . $$
The Euler-Lagrange equation for $\mathcal L_G$ is given in \cite[Equation (48)]{gauduchon}. For compact complex surfaces ($n=2$), it turns out that the critical points of $\mathcal L_G$ are precisely the K\"ahler metrics  ({\itshape i.e.} $\theta_{\Omega}=0$) \cite[Th\'eor\`eme III.4]{gauduchon}, which are in fact absolute minima.
In \cite[Section V]{gauduchon}, the restriction of $\mathcal L_G$ to a conformal class $\{\Omega\}_1:=\{\Omega\}\cap\mathcal{H}_1$ in $\mathcal H_1$ is also considered, and the critical points $\Omega$ are characterized by the equation  $|\theta_\Omega|_\Omega^2+2d^{*_\Omega}\theta_\Omega=k\in\RR$ in \cite[Proposition at page 516]{gauduchon}. 

\item The {\em Vaisman functional} \cite{vaisman}:
$$ \mathcal U_V \colon \mathcal{H}_1 \to \mathbb R, \qquad
\mathcal U_V(\Omega) := \int_X |d\theta_\Omega|_\Omega^2 \frac{\Omega^n}{n!} = \int_X d\theta_\Omega \wedge *_\Omega d\theta_\Omega $$
is conformally invariant.
On compact complex surfaces, it rewrites as $\mathcal U_V(\Omega) = -\int_X d\theta_\Omega \wedge Jd\theta_\Omega$. In this case, its Euler-Lagrange equation is given in \cite[Theorem 2.1, Equation (2.15)]{vaisman}, and its second variation is computed in \cite[Equation (3.10)]{vaisman}. Clearly, locally conformally K\"ahler metrics for $n=2$,  and for arbitrary $n$, locally conformally balanced metrics are critical points and absolute minima of $\mathcal U_V$.
\end{itemize}

In the present note, we investigate a few functionals of similar type, restricted to a conformal class of normalized almost Hermitian metrics $\{\Omega\}_1\subset\mathcal H_1$. The first functional we consider is given by:
$$
\mathcal{G} \colon \mathcal H_1 \mapsto \int_X |d^{*_\Omega} \theta_\Omega|_\Omega^2 \frac{\Omega^n}{n!}.
$$
Clearly, Gauduchon metrics are critical points, and in fact absolute minima, for $\mathcal G$. Our  first main result states that these are all the critical points of $\mathcal G$:

\renewcommand{\referenza}{\ref{thm:critical-g}}
\begin{thm*}
Let $(X,J)$ be a $2n$-dimensional compact manifold endowed with an almost complex structure, where $n\geq2$. Let $\Omega\in \mathcal H_1$ and let $\{\Omega\}_1\subset \mathcal H_1$ denote the set of almost Hermitian metrics of total volume one which are conformal to $\Omega$. Then the following are equivalent:
\begin{enumerate}
\item $\Omega$ is a critical point for $\mathcal G$ on $\mathcal H_1$;
\item $\Omega$ is a critical point for the restriction of $\mathcal G$ to $\{\Omega\}_1$;
\item $\Omega$ is Gauduchon.
\end{enumerate}
\end{thm*}
\noindent Together with Gauduchon's result, this implies that $\mathcal G$ has exactly one critical point in every $\{\Omega\}_1\subset\mathcal H_1$.

Another functional we study is given by: 
$$
\mathcal{F} \colon \{\Omega\}_1 \mapsto \int_X |dJ\theta_\Omega|_\Omega^2 \frac{\Omega^n}{n!}.
$$
The main motivation for considering this functional comes from the fact that the form $dJ\theta_{\Omega}$ is related to the curvature of natural connections on the manifold. More precisely, it measures the difference between the Ricci curvatures of the Chern connection and the Bismut connection of the Hermitian metric $\Omega$ (see Remark~\ref{rmk:motivation-f}). Furthermore, in the context of lcK metrics, $dJ\theta_\Omega$ is, up to a constant multiple, the Chern curvature of the naturally associated weight line bundle to an lcK metric. 

The critical points of $\mathcal F$
are described in Proposition \ref{prop:critical-f} by the equation:
$$ (n-2)|dJ\theta_\Omega|_\Omega^2-2(n-1)(dd^c)^{*_\Omega}\Delta_\Omega\Omega = k\in\RR.$$
When $n=2$, this is equivalent to the equation:
\begin{equation*}
dd^c\Delta_\Omega\Omega=0
\end{equation*}
and when $J$ is moreover integrable, this is  further equivalent to:
\begin{equation*}
dd^c(\Lambda_{\Omega}(dJ\theta_\Omega)\Omega)=0.
\end{equation*}
We show in Corollary~\ref{cor:minF} that in this case, the critical points of $\mathcal F$ exist, are unique, and minimize $\mathcal F$.

Motivated by this result, we introduce a new class of almost Hermitian metrics, which we call \textit{distinguished}, defined by the condition:
\begin{equation*}
dd^c(f^{n-1}_\Omega\Omega^{n-1})=0, \quad \text{ where } f_\Omega:=|\theta_\Omega|^2_\Omega+d^{*_\Omega}\theta_\Omega=-\Lambda_{\Omega}(dJ\theta_\Omega).
\end{equation*}
These metrics are not in general Gauduchon (see Corollary~\ref{cor:f-2-vaisman2}). However,  similarly to Gauduchon metrics, they give rise to canonical representatives in a given conformal class:

\renewcommand{\referenza}{\ref{thm:criticalF}}
\begin{thm*}
Let $(X,J)$ be a compact almost complex manifold of real dimension $2n>2$, and let $\{\Omega\}_1$ be a conformal class of normalized almost Hermitian metrics. Then there exists and is unique a distinguished metric $\Omega\in\{\Omega\}_1$. This metric is either balanced, i.e. $\theta_\Omega=0$, or is characterized by the property that $f_\Omega$ is strictly positive on $X$ and the metric $f_\Omega\Omega$ is Gauduchon.
\end{thm*}

Additionally, we discuss the functional:
$$
\mathcal{A} \colon\{\Omega\}_1\mapsto \int_X | d \Omega |_\Omega^2 \frac{\Omega^n}{n!}.
$$
Clearly almost K\"ahler metrics are minimizers for $\mathcal A$. We note that $\mathcal A(\Omega)$ also coincides with the $L^2$-norm of $d^c\Omega$, which is the torsion of the Bismut connection (see Remark~\ref{rmk: motivation-A}). This gives an alternative motivation for studying this functional.  

The critical points of $\mathcal{A}$ are described, in Proposition~\ref{prop:crit-points-A}, by the equation:
$$
(n-1) | d\Omega |_{\Omega}^2 + 2 d^{*_\Omega} \theta_\Omega = k\in\RR.
$$
In particular, if $\{\Omega\}_1$ is an lcK conformal class or if $n=2$, the functional $\mathcal A$ is a constant multiple of the functional $\mathcal L_G$ (see Remark~\ref{remdOmega}). It follows by \cite[Théorème~III.4]{gauduchon} that the critical points of $\mathcal A$ viewed on $\mathcal H_1$, for $n=2$, are precisely the K\"ahler metrics. 

The last functional we discuss is given by:
\begin{equation*}
\mathcal{R} \colon\{\Omega\}_1\mapsto \int_X | dd^c \Omega |_\Omega^2 \frac{\Omega^n}{n!}.
\end{equation*}
By definition, SKT metrics minimize $\mathcal R$. In Proposition~\ref{prop:r-critical} we show that the critical points of this functional are characterized by the equation:
\begin{equation*}
(n-4)|dd^c\Omega|^2_\Omega+2\Lambda_\Omega(dd^c)^{*_\Omega}dd^c\Omega=k\in\RR.
\end{equation*}
Furthermore, in the case $n=2$, the critical points of $\mathcal R$ exist, are unique and are precisely the Gauduchon metrics (see Corollary~\ref{cor:critR}). In particular, they minimize $\mathcal R$.

In the last section of this note, we do some explicit computations related to these functionals on the Inoue-Bombieri surfaces of type $S_M$. They show that, generally, the infimum of the functionals $\mathcal F$ and $\mathcal A$  on $\mathcal H_1$ can be $0$, without it being attained.

\break

\def\arraystretch{1.5}
\begin{table}[h!]
{\resizebox{1.05\textwidth}{!}{
\begin{tabular}{c||c|c|c}
{\bfseries functional} & \multicolumn{2}{c|}{\bfseries critical points}&\multirow{2}{*}{\bfseries reference}\\
{\bfseries restricted to $\{\Omega\}_1$} & {\bfseries complex surfaces} & {\bfseries almost complex, $\dim X\geq 4$}&  \\
\hline\hline
$\mathcal L_G(\Omega) = \int_X |\theta_\Omega|_\Omega^2 \frac{\Omega^n}{n!}$ & K\"ahler ($!$, $\min$) & $|\theta_\Omega|_\Omega^2+2d^{*_\Omega}\theta_\Omega=k$ & \cite{gauduchon} \\
\hline
$U_V(\Omega) := \int_X |d\theta_\Omega|_\Omega^2 \frac{\Omega^n}{n!}$ & \multicolumn{2}{c|}{conformal invariant} & \cite{vaisman} \\
\hline\hline
$\mathcal{G}(\Omega) = \int_X |d^{*_\Omega} \theta_\Omega|_\Omega^2 \frac{\Omega^n}{n!}$ & \multicolumn{2}{c|}{Gauduchon ($\exists!$, $\min$)} & Thm \ref{thm:critical-g} \\
\hline
$\mathcal{F}(\Omega) = \int_X |dJ\theta_\Omega|_\Omega^2 \frac{\Omega^n}{n!}$ & distinguished ($\exists!$, $\min$) ($dd^c\Delta_\Omega\Omega=0$) & $(n-2)|dJ\theta_\Omega|_\Omega^2-2(n-1)(dd^c)^{*_\Omega}\Delta_\Omega\Omega = k$ & Cor \ref{cor:minF}, Prop \ref{prop:critical-f} \\
\hline
$\mathcal{A}(\Omega) = \int_X | d \Omega |_\Omega^2 \frac{\Omega^n}{n!}$ & K\"ahler ($!$, $\min$) & $(n-1) | d\Omega |_{\Omega}^2 + 2 d^{*_\Omega} \theta_\Omega = k$ & Prop \ref{prop:crit-points-A}, \cite{gauduchon} \\
\hline
$\mathcal{R} (\Omega) = \int_X | dd^c \Omega |_\Omega^2 \frac{\Omega^n}{n!}$ & Gauduchon/SKT ($\exists!$, $\min$) & $(n-4)|dd^c\Omega|^2_\Omega+2\Lambda_\Omega(dd^c)^{*_\Omega}dd^c\Omega=k$ & Prop \ref{prop:r-critical} \\
\hline\hline
\end{tabular}
}}
\caption{Summary of the critical points of the functionals. (Here, $k$ denotes a real constant.)}
\label{table:name}
\end{table}

\subsection*{Notation}
Given an almost complex manifold $(X,J)$, we denote by $\wedge^kX$ and by $\wedge^{p,q}X$ the spaces of smooth $k$-forms and $(p,q)$-forms, respectively, on $X$.   We extend $J$ to act as an isomorphism on $\wedge^{p,q}X$ by $J \alpha = \sqrt{-1}^{q-p}\alpha$, $\alpha\in\wedge^{p,q}X$, following \cite[(2.10)]{besse}. Via this extension, it follows that $J^2=(-1)^k\mathrm{id}$, so that $J^{-1}=(-1)^kJ=J^*$ on $\wedge^kX$, where $J^*$ is the pointwise adjoint of $J$ with respect to some (and so any) Hermitian metric. We denote by $d^c$ the differential operator $d^c:=-J^{-1}dJ$. For a Hermitian metric $\Omega$ and a one-form $\alpha$, we denote by $\alpha^{\sharp_\Omega}$ the vector field which is metric dual to $\alpha$ with respect to $\Omega$. For a vector field $V$, we denote by $\iota_V$ the interior product with $V$.

\section{The functional $\mathcal{G}$}

Let $(X,J)$ be a compact almost complex manifold, let $\mathcal H_1$ denote the set of almost Hermitian metrics on $(X,J)$ of total volume $1$, and let 
\begin{equation*}
\{\Omega\}_1:= \left\{ \phi\cdot \Omega \;:\; \phi\in\ce(X,\RR), \phi>0, \int_X \phi^n\frac{\Omega^n}{n!} =1\right\} \subseteq \mathcal H_1
\end{equation*} 
denote a conformal class of normalized Hermitian metrics. We consider the following functional:
\begin{equation}\label{eq:g}
\mathcal G \colon \{\Omega\}_1 \to \mathbb R, \qquad
\mathcal G(\Omega) := \int_X ( d^{*_\Omega} \theta_\Omega )^2 \frac{\Omega^n}{n!} = - \int_X d^{*_\Omega} \theta_\Omega \cdot d *_\Omega \theta_\Omega.
\end{equation}

\begin{thm}\label{thm:critical-g}
Let $n\geq 2$, let $(X,J)$ be a $2n$-dimensional compact manifold endowed with an almost complex structure and let $\{\Omega\}_1$ denote a conformal class of almost Hermitian metrics with total volume one. Then $\Omega\in \{\Omega\}_1$ is a critical point of the functional $\mathcal{G}$ defined by \eqref{eq:g} if and only if $\Omega$ is a Gauduchon metric.
\end{thm}

\begin{proof}
A Gauduchon metric $\Omega\in \{\Omega\}_1$ satisfies $d^{*_\Omega}\theta_\Omega=0$, so it is clearly a minimum point for $\mathcal{G}$.

Conversely, let us first determine the equation that $\Omega\in \{\Omega\}_1$ has to satisfy in order to be a critical point for $\mathcal G$.
We first observe some general formulas. For an almost Hermitian metric $\Omega$ on $X^{2n}$ and for $\tilde\Omega=\phi\cdot \Omega$ in its conformal class, with $\phi\in\ce(X,\RR)$, $\phi>0$, one has:
\begin{eqnarray*}
\mathrm{Vol}_{\tilde\Omega}(X) &:=&\int_X\frac{\tilde\Omega^n}{n!}= \int_X \phi^n \frac{\Omega^n}{n!}, \\
\theta_{\tilde\Omega} &=& \theta_\Omega+(n-1)d\log\phi, \\
*_{\tilde\Omega}\lfloor_{\wedge^kX} &=& \phi^{n-k} *_{\Omega}\lfloor_{\wedge^kX} .
\end{eqnarray*}

Consider a variation of $\Omega$ in $\{\Omega\}_1$,
$$ \tilde\Omega_t = \phi_t \cdot \Omega , $$
where, for $|t|$ small enough, we write:
\begin{equation*}
\phi_t=1+t\dot\phi+o(t), \ \ \dot\phi\in\ce(X,\RR).
\end{equation*}
The condition $\mathrm{Vol}_{\tilde\Omega}(X)=1$ implies that the function $\dot\phi$ must satisfy:
\begin{equation}\label{varOm} 
\int_X \dot\phi \frac{\Omega^n}{n!}=0.
\end{equation}

We compute:
\begin{eqnarray*}
\mathcal G(\tilde\Omega_t) &=&
\int_X \left(d^{*_{\tilde\Omega_t}}\theta_{\tilde\Omega_t}\right)^2 \frac{\tilde\Omega_t^n}{n!}
= \int_X \left|d *_{\tilde\Omega_t}\theta_{\tilde\Omega_t}\right|_{\tilde\Omega_t}^2 \frac{\tilde\Omega_t^n}{n!} \\
&=& \int_X \left|d \left( \phi_t^{n-1} *_{\Omega}\left(\theta_{\Omega}+(n-1)\phi_t^{-1}d\phi_t\right)\right) \right|_{\Omega}^2 \phi_t^{-2n+n} \frac{\Omega^n}{n!} \\
&=& \int_X \left|d (\phi_t^{n-1} *_{\Omega} \theta_{\Omega}) + (n-1)d(\phi_t^{n-2} *_{\Omega} d\phi_t) \right|_{\Omega}^2 \phi_t^{-n} \frac{\Omega^n}{n!} \\
&=& \int_X \left|d (\phi_t^{n-1} *_{\Omega} \theta_{\Omega})\right|_\Omega^2 \phi_t^{-n} \frac{\Omega^n}{n!}
+ (n-1)^2\int_X \left|d(\phi_t^{n-2} *_{\Omega} d\phi_t) \right|^2_{\Omega} \phi_t^{-n} \frac{\Omega^n}{n!} \\
&& + 2(n-1)\int_X \left\langle d (\phi_t^{n-1} *_{\Omega} \theta_{\Omega}) \middle\vert d(\phi_t^{n-2} *_{\Omega} d\phi_t) \right\rangle_\Omega^2 \phi_t^{-n} \frac{\Omega^n}{n!} \\
&=&\int_X\left|d\left( (1+t(n-1)\dot\phi)*_\Omega\theta_\Omega\right)\right|_\Omega^2(1-tn\dot\phi)\frac{\Omega^n}{n!}\\
&&+(n-1)^2\int_X \left|d\left((1+t(n-2)\dot\phi)*_{\Omega} td\dot\phi)\right) \right|^2_{\Omega} (1-tn\dot\phi)\frac{\Omega^n}{n!} \\
&& + 2(n-1)\int_X \left\langle d\left( (1+t(n-1)\dot\phi) *_{\Omega} \theta_{\Omega})\right) \middle\vert d\left((1+t(n-2)\dot\phi) *_{\Omega} td\dot\phi\right) \right\rangle_\Omega (1-tn\dot\phi)\frac{\Omega^n}{n!}+o(t) \\
&=&\mathcal G(\Omega)-tn\int_X\left(d^{*_\Omega}\theta_\Omega \right)^2\dot\phi\frac{\Omega^n}{n!}+2t(n-1)\int_X\left\langle d*_\Omega\theta_\Omega\middle\vert d(\dot\phi *_\Omega\theta_\Omega)\right\rangle_\Omega\frac{\Omega^n}{n!}\\
&&+2t(n-1)\int_X\left\langle d*_\Omega\theta_\Omega\middle\vert d*_\Omega d\dot\phi\right\rangle_\Omega\frac{\Omega^n}{n!}+o(t).\\
\end{eqnarray*}
For the third term in the above equation we find:
\begin{eqnarray*}
\int_X\left\langle d*_\Omega\theta_\Omega\middle\vert d(\dot\phi *_\Omega\theta_\Omega)\right\rangle_\Omega\frac{\Omega^n}{n!}&=&\int_X\left\langle-*_\Omega d*_\Omega d*_\Omega\theta_\Omega\middle\vert *_\Omega\theta_\Omega\right\rangle_\Omega\dot\phi\frac{\Omega^n}{n!}\\
&=& \int_X\left\langle dd^{*_\Omega}\theta_\Omega\middle\vert \theta_\Omega\right\rangle_\Omega\dot\phi\frac{\Omega^n}{n!}
\end{eqnarray*}
while for the fourth term we find:
\begin{eqnarray*}
\int_X\left\langle d*_\Omega\theta_\Omega\middle\vert d*_\Omega d\dot\phi\right\rangle_\Omega\frac{\Omega^n}{n!}&=&-\int_X\left\langle d*_\Omega d*_\Omega\theta_\Omega\middle\vert d\dot\phi\right\rangle_\Omega\frac{\Omega^n}{n!}\\
&=&\int_X \Delta_\Omega d^{*_\Omega}\theta_\Omega\cdot\dot\phi\frac{\Omega^n}{n!}.
\end{eqnarray*}

Now if $\Omega$ is a critical point of $\mathcal G$, then $\left.\frac{d}{dt}\right\lfloor_{t=0}\mathcal G(\tilde\Omega_t)=0$ for any $\tilde \Omega_t=\phi_t\Omega$ defined by \eqref{varOm}. The above computations thus imply the following equation for $\Omega$:
\begin{equation}\label{critGG}
\Delta_\Omega d^{*_\Omega} \theta_{\Omega} +\langle dd^{*_\Omega}\theta_\Omega\vert\theta_\Omega\rangle_\Omega- \frac{n}{2(n-1)}(d^{*_\Omega} \theta_{\Omega})^2 = k,\ \ \ k\in\mathbb{R}.
\end{equation}
Let us denote by $f:=d^{*_\Omega}\theta_\Omega$. Integrating \eqref{critGG} over $X$ with respect to $\Omega$, we find:
\begin{align*}
k =&\int_X\Delta_\Omega f\frac{\Omega^n}{n!}+\int_Xf^2\frac{\Omega^n}{n!}-\frac{n}{2(n-1)}\int_Xf^2\frac{\Omega^n}{n!}\\
=&\frac{n-2}{2(n-1)}\int_Xf^2\frac{\Omega^n}{n!}\geq 0.
\end{align*}
Let now $x_{\text{min}} \in X$ be a minimum point of $f$, which exists by the compactness of $X$. It follows that $df_{x_{\text{min}}}=0$ and $\Delta_\Omega f(x_{\text{min}})\leq 0$. Moreover, as the integral of $f$ over $X$ vanishes, we have $f(x_{\text{min}})\leq 0$, with equality if and only if $f=0$ everywhere on $X$. We thus find, via \eqref{critGG}:
\begin{equation*}
0\leq k=\Delta_\Omega f(x_{\text{min}})-\frac{n}{2(n-1)}f^2(x_{\text{min}})\leq 0
\end{equation*}
implying that $\min f=f(x_{\text{min}})=0$, or also that $f$ vanishes everywhere. Thus the critical point $\Omega$ is indeed a Gauduchon metric.
\end{proof}

\begin{rmk}
By \cite[Th\'eor\`eme~1]{gauduchon-CRAS}, it follows that in any normalized conformal class $\{\Omega\}_1$, a critical point for $\mathcal G$ exists and is unique. 
\end{rmk}

\begin{cor}\label{cor:critG}
On a compact almost complex manifold $(X^{2n},J)$ with $n\geq 2$, a Hermitian metric $\Omega\in\mathcal H_1$ is critical for $\mathcal G$ on $\mathcal H_1$ if and only if $\Omega$ is a Gauduchon metric. 
\end{cor}
\begin{proof}
Clearly Gauduchon metrics are critical points of $\mathcal G$ on $\mathcal H_1$. Conversely, if $\Omega$ is a critical point for $\mathcal G$ on $\mathcal H_1$, then it is also a critical point for $\mathcal G$ on $\{\Omega\}_1$, hence $\Omega$ is Gauduchon.
\end{proof}

\section{The functional $\mathcal F$}
We consider now the following functional:
\begin{equation}\label{eq:f}
\mathcal F \colon \mathcal{H}_1 \to \mathbb R, \qquad
\mathcal F(\Omega) := \int_X |dJ\theta_\Omega|_\Omega^2 \frac{\Omega^n}{n!} = \int_X dJ\theta_\Omega \wedge *_\Omega dJ\theta_\Omega ,
\end{equation}
as well as its restriction to a conformal class in $\mathcal H_1$.

\begin{rmk}\label{rmk:motivation-f}
One motivation for studying the $L^2$-norm of the form $dJ\theta_\Omega$ is the following. Recall that, on a Hermitian manifold $(X,\Omega)$, there is a {\em canonical family} $\{\nabla^t\}_{t\in\mathbb R}$ of Hermitian connections, that we will call {\em distinguished connections} in the sense of \cite{libermann}, see also \cite{gauduchon-bumi}. It consists of the affine line through the {\em Chern connection} (namely, the unique Hermitian connection whose $(0,1)$-part is the Cauchy-Riemann operator $\overline\partial$) and the {\em Bismut connection} (namely, the unique Hermitian connection whose torsion is totally skew-symmetric), and it is characterized by prescribed torsion $T^{\nabla^t} = \frac{1}{4} (3t-1) (d^c \Omega)^+ - \frac{1}{4}(t+1) \mathfrak{M}(d^c\Omega)^+$, see \cite{gauduchon-bumi} for notation and more details.
We recall that the Ricci form of the Chern connection can be locally expressed as $\mathrm{Ric}^{Ch}(\Omega)\stackrel{\text{loc}}{=}\sqrt{-1}\,\overline\partial\partial\log\det \Omega$ and represents the first Bott-Chern class in Bott-Chern cohomology, $c_1^{BC}(X)=[\mathrm{Ric}^{Ch}(\Omega)]\in H^{1,1}_{BC}(X)\to H^2_{dR}(X)\ni c_1(X)$. By  \cite[Formula (8)]{grantcharov-grantcharov-poon}, the Ricci form of any distinguished connection $\nabla^t$ is expressed in terms of the Ricci form of the Chern connection as:
$$ \mathrm{Ric}^t = \mathrm{Ric}^{Ch}+\frac{1-t}{2} dJ\theta .$$
\end{rmk}

\begin{prop}\label{prop:critical-f}
On a compact almost complex $2n$-dimensional manifold $(X,J)$, the critical points $\Omega$ for the functional $\mathcal F$ in \eqref{eq:f} restricted to the conformal class $\{\Omega\}_1$ are described by the equation
\begin{equation}\label{eq:critical-F-ndim}
(n-2)|dJ\theta_\Omega|_\Omega^2-2(n-1)(dd^c)^{*_\Omega}\Delta_\Omega\Omega = k
\end{equation}
for $k$ constant.
\end{prop}

\begin{proof}
We look for the conditions to be satisfied by $\Omega$ in order to be a critical point of $\mathcal F$ in its volume-normalized conformal class $\{\Omega\}_1 \subseteq \mathcal H_1$.
Consider a variation of $\Omega$ in its conformal class,
$$ \tilde\Omega_t = \phi_t \cdot \Omega, \qquad\phi_t=1+t\dot\phi+o(t) $$
for $|t|$ small enough and $\dot\phi\in\ce(X,\RR)$ satisfying, as before,
$$ \int_X \dot\phi \frac{\Omega^n}{n!}=0. $$

We compute:
\begin{eqnarray*}
\mathcal F(\tilde\Omega_t) &=& \int_X \left|dJ\theta_{\tilde\Omega_t}\right|_{\tilde\Omega_t}^2 \frac{\tilde\Omega_t^n}{n!}
= \int_X \left| dJ\theta_{\tilde\Omega_t} \right|_{\Omega}^2 \phi_t^{n-2} \frac{\Omega^n}{n!} \\
&=& \int_X \left| dJ\theta_\Omega+(n-1)d(\phi_t^{-1}Jd\phi_t) \right|_{\Omega}^2 \phi_t^{n-2} \frac{\Omega^n}{n!} \\
&=& \int_X \left| dJ\theta_\Omega+(n-1)d\phi_t^{-1}\wedge Jd\phi_t +(n-1) \phi_t^{-1} dJd\phi_t \right|_{\Omega}^2 \phi_t^{n-2} \frac{\Omega^n}{n!} \\
&=& \mathcal{F}(\Omega) + 2 (n-1)\int_X \left\langle dJ\theta_\Omega \middle\vert d\phi_t^{-1}\wedge Jd\phi_t + \phi_t^{-1} dJd\phi_t \right\rangle_{\Omega} \frac{\Omega^n}{n!} \\
&& + t (n-2) \int_X |dJ\theta_\Omega|_\Omega^2 \dot\phi \frac{\Omega^n}{n!} + o(t) \\
&=& \mathcal{F}(\Omega) + t \left( 2(n-1)\int_X \left\langle dJ\theta_\Omega \middle\vert  dJd\dot\phi \right\rangle_{\Omega}\frac{\Omega^n}{n!} + (n-2) \int_X |dJ\theta_\Omega|_\Omega^2 \dot\phi \frac{\Omega^n}{n!} \right) + o(t)  \\
\end{eqnarray*}
whence we get:
\begin{eqnarray*}
\left.\frac{d}{dt}\right\lfloor_{t=0} \mathcal F(\Omega_t) &=& 2(n-1)\int_X \left\langle dJ\theta_\Omega \middle\vert  dd^c\dot\phi \right\rangle_{\Omega}\frac{\Omega^n}{n!} + (n-2) \int_X |dJ\theta_\Omega|_\Omega^2 \dot\phi \frac{\Omega^n}{n!} \\
&=& \int_X \left( (n-2)|dJ\theta_\Omega|_\Omega^2+2(n-1)(dd^c)^{*_\Omega} dJ\theta_\Omega \right) \dot\phi \frac{\Omega^n}{n!} .
\end{eqnarray*}
Therefore $\Omega$ is a critical point for $\mathcal F$ restricted to its conformal class if and only if:
\begin{equation}\label{extrF} 
(n-2)|dJ\theta_\Omega|_\Omega^2+2(n-1)(dd^c)^{*_\Omega} dJ\theta_\Omega = k
\end{equation}
for $k$ constant.
We rewrite the second term in the last equation as follows, by using \eqref{eq:theta-def}:
\begin{eqnarray*}
(dd^c)^{*_\Omega}dJ\theta_\Omega &=&(d^c)^{*_\Omega}d^{*_\Omega}(-dd^{*_\Omega}\Omega)\\
&=&-(d^c)^{*_\Omega}d^{*_\Omega}\Delta_\Omega\Omega\\
&=&-(dd^c)^{*_\Omega}\Delta_\Omega\Omega.
\end{eqnarray*}

Therefore $\Omega$ is a critical point for $\mathcal F$ restricted to its conformal class if and only if:
$$ (n-2)|dJ\theta_\Omega|_\Omega^2-2(n-1)(dd^c)^{*_\Omega}\Delta_\Omega\Omega = k $$
for $k$ constant, proving the statement.
\end{proof}

\begin{cor}\label{cor:f-2}
On a compact almost complex $4$-dimensional manifold, the critical points $\Omega$ for the functional $\mathcal F$ in \eqref{eq:f} restricted to the conformal class $\{\Omega\}_1$ in $\mathcal{H}_1$ are described by the equation:
\begin{equation}\label{eq:critical-F-surface}
dd^c\Delta_\Omega\Omega=0.
\end{equation}
\end{cor}

\begin{proof} By \eqref{eq:critical-F-ndim}, the critical points are described by the equation $(dd^c)^{*_\Omega}\Delta_\Omega\Omega=0$ and since on a complex surface $*_\Omega\Omega=\Omega$, this further reads as $dd^c\Delta_{\Omega}\Omega=0$.
\end{proof}

In the integrable case, this is further equivalent to:

\begin{cor}\label{cor:critF-surface}
On a compact complex surface, the critical points for the functional $\mathcal F$ restricted to a conformal class $\{\Omega\}_1$ in $\mathcal H_1$ are characterized by the equation:
\begin{equation}\label{eq:traceSurf}
dd^c(f\Omega)=0, \quad f=\Lambda_{\Omega}(-dJ\theta_\Omega)=|\theta_\Omega|_\Omega^2+d^{*_\Omega}\theta_\Omega.
\end{equation}
\end{cor}
\begin{proof}

We find that eq.~\eqref{extrF} is equivalent to: 
\begin{eqnarray*}
 *_{\Omega}dd^c*_{\Omega}dJ\theta_{\Omega}=0 &\Leftrightarrow  &dd^c *_{\Omega} dJ\theta_{\Omega}=0.
\end{eqnarray*} 

On the other hand, by \cite[Formula (50)]{gauduchon}, we have:
\begin{equation}\label{trace2}
*_{\Omega}dJ\theta_{\Omega} = \Lambda_\Omega(dJ\theta_{\Omega})\Omega-dJ\theta_{\Omega}^{1,1}+dJ\theta_{\Omega}^{(2,0)+(0,2)}.
\end{equation}
Using that $\dim_{\mathbb R} X=4$, we obtain $dd^c(dJ\theta_{\Omega}^{(2,0)+(0,2)})=0$, and since $J$ is integrable we have $0=dd^c dJ\theta_{\Omega}=dd^c (dJ\theta_{\Omega}^{1,1})$. Thus we have shown that eq. \eqref{eq:critical-F-surface} is equivalent to:
\begin{equation*}
dd^c (\Lambda_\Omega(dJ\theta_{\Omega})\Omega)=0.
\end{equation*}

Next, we show that:
\begin{equation}\label{tddJtheta} 
\Lambda_{\Omega}(dJ\theta_{\Omega})=-|\theta_{\Omega}|^2-d^{*_{\Omega}}\theta_{\Omega}.
\end{equation}
Indeed, for any function $\phi\in\ce(X,\RR)$, we have:
\begin{eqnarray*}
\int_X\Lambda_{\Omega}(dJ\theta_{\Omega})\phi\frac{\Omega^2}{2}&=&\int_X\langle \phi dJ\theta_{\Omega}\vert\Omega\rangle_{\Omega}\frac{\Omega^2}{2}\\
&=&\int_X\langle d(\phi J\theta_{\Omega})-d\phi\wedge J\theta_{\Omega}\vert\Omega\rangle_{\Omega}\frac{\Omega^2}{2}\\
&=& \int_X \langle \phi J\theta_{\Omega}\vert-J\theta_{\Omega}\rangle_{\Omega}\frac{\Omega^2}{2}+\int_X\langle d\phi\vert \iota_{J\theta_{\Omega}^\sharp}\Omega\rangle_{\Omega}\frac{\Omega^2}{2}\\
&=&-\int_X\phi|\theta_{\Omega}|^2\frac{\Omega^2}{2}-\int_{X}\phi d^{*_{\Omega}}\theta_{\Omega}\frac{\Omega^2}{2}
\end{eqnarray*}
from which eq. \eqref{tddJtheta} follows. This concludes the proof.
\end{proof}

Let us note that eq. \eqref{tddJtheta} holds on any compact almost complex manifold.  In view of the above corollary, we introduce the following definition:

\begin{defi}
Let $(X,J)$ be an almost complex manifold and let $\Omega$ be an almost Hermitian metric. We call $\Omega$ \emph{distinguished} if it satisfies the equation:
\begin{equation}\label{eq:CritFf}
dd^c(f^{n-1}_{\Omega}\Omega^{n-1})=0, \quad \text{ where } f_{\Omega} :=\Lambda_{\Omega}(-dJ\theta_\Omega)=|\theta_\Omega|_{\Omega}^2+d^{*_\Omega}\theta_\Omega.
\end{equation}
\end{defi}

Like Gauduchon metrics, these metrics can be seen as canonical representatives in a given normalized conformal class :

\begin{thm}\label{thm:criticalF}
Let $(X,J)$ be a compact almost complex manifold of real dimension $2n>2$, and let $\{\Omega\}_1$ be a conformal class of normalized almost Hermitian metrics. Then there exists and is unique a distinguished metric $\Omega\in\{\Omega\}_1$. This metric is either balanced, i.e. $\theta_\Omega=0$, or is characterized by the property that $f_{\Omega}:=|\theta_\Omega|^2_\Omega+d^{*_\Omega}\theta_\Omega$ is strictly positive on $X$ and the metric $f_{\Omega}\Omega$ is Gauduchon.
\end{thm}
\begin{proof}
Let us fix $\Omega\in\{\Omega\}_1$ and take any smooth function $q\in\ce(X)$. First, we compute:
\begin{eqnarray*}
dd^c(q\Omega^{n-1})&=&d(d^cq\wedge\Omega^{n-1}+qJ\theta_{\Omega}\wedge\Omega^{n-1})\\
&=&dd^cq\wedge\Omega^{n-1}-d^cq\wedge\theta\wedge\Omega^{n-1}+dq\wedge J\theta_{\Omega}\wedge\Omega^{n-1}\\
&&+qdJ\theta_{\Omega}\wedge\Omega^{n-1}+q\theta_{\Omega}\wedge J\theta_{\Omega}\wedge\Omega^{n-1}\\
&=&\Lambda_{\Omega}\left(dd^cq-d^cq\wedge\theta_{\Omega}+dq\wedge J\theta_{\Omega}+qdJ\theta_{\Omega}+q\theta_{\Omega}\wedge J\theta_{\Omega}\right)\frac{\Omega^n}{n}.
\end{eqnarray*}

Using that for any function $q\in\ce(X)$, 
\begin{equation}\label{eq:ddc}
\Lambda_{\Omega}(dd^cq)=-\Delta_{\Omega} q-\langle dq |\theta_{\Omega}\rangle_{\Omega}
\end{equation}
which is proven in the same manner as eq. \eqref{tddJtheta}, we find that $dd^c(q\Omega^{n-1})=0$ is equivalent to:
\begin{eqnarray*} 
0&=&\Lambda_{\Omega}\left(dd^cq-d^cq\wedge\theta_{\Omega}+dq\wedge J\theta_{\Omega}+qdJ\theta_{\Omega}+q\theta_{\Omega}\wedge J\theta_{\Omega}\right)\\
&=&-\Delta_{\Omega} q-\langle dq\vert \theta_{\Omega}\rangle_{\Omega}+2\langle dq\vert \theta_{\Omega}\rangle_{\Omega}-qf_{\Omega}+q|\theta_{\Omega}|_{\Omega}^2
\end{eqnarray*}
or also to:
\begin{equation}\label{eqCritF2}
U_\Omega q:=\Delta_{\Omega} q-\langle dq\vert \theta_{\Omega}\rangle_{\Omega}+qd^{*_{\Omega}}\theta_{\Omega}=0
\end{equation}
meaning that $q$ is in the kernel of the elliptic operator $U_\Omega$.

However, by the proof of \cite[Théorème at page~502]{gauduchon} using the Hopf maximum principle \cite{hopf}, the kernel of the operator $U_\Omega$ is one dimensional and any $0\neq q\in \ker U_\Omega$ is either strictly positive or strictly negative. Furthermore, for a strictly positive function $q$ on $X$, $q\in\ker U_\Omega$ if and only if $q^{1/(n-1)}\Omega$ is Gauduchon. 

We infer that if eq.~\eqref{eqCritF2} is satisfied for $q=f^{n-1}_\Omega$, then either $f_{\Omega}=0$ and $\Omega$ is balanced, or $f_\Omega\neq 0$. In the second case, $f_{\Omega}$ has constant sign and since $\int_Xf_{\Omega}\frac{\Omega^n}{n!}>0$, then $f_{\Omega}>0$ everywhere and $f_{\Omega}\Omega$ is Gauduchon. Conversely, if $f_{\Omega}=0$,
 or $f_{\Omega}>0$ and $f_{\Omega}\Omega$ is Gauduchon, then eq. \eqref{eq:CritFf} holds. This concludes the proof of the last statement in the theorem.

In order to show existence and uniqueness of solutions to eq.~\eqref{eq:CritFf}  in a given normalized conformal class, let us fix $\Omega_0$ a Gauduchon metric on $(X,J)$. By the above, we need to show that there exists a unique function $\varphi\in\ce(X,\RR)$ and constant $k\in\RR$, $k\geq 0$ such that $\Omega:=e^{\varphi}\Omega_0 \in \{\Omega_0\}_1$ and $e^{\varphi}f_{\Omega}= k$.

Using that $\theta_{\Omega}=\theta_0+(n-1)d\varphi$ and that for any $\alpha\in\wedge^1X$, we have:
\begin{equation*}\label{eq:confd}
d^{*_{\Omega}}\alpha=e^{-\varphi}\left(d^{*_{\Omega_0}}\alpha-(n-1)\langle d\varphi\vert \alpha\rangle_{\Omega_0}\right),
\end{equation*}
we compute:
\begin{eqnarray*}
f_{\Omega} &=&e^{-\varphi}( (n-1)\Delta_{\Omega_0}\varphi-(n-1)\langle d\varphi\vert\theta_0\rangle_{\Omega_0}-(n-1)^2|d\varphi|^2_{\Omega_0}\\
&&+|\theta_0|^2_{\Omega_0}+(n-1)^2|d\varphi|^2_{\Omega_0}+2(n-1)\langle d\varphi\vert\theta_0\rangle_{\Omega_0})\\
&=&e^{-\varphi}\left( (n-1)\Delta_{\Omega_0}\varphi+(n-1)\langle d\varphi\vert\theta_0\rangle_{\Omega_0}+|\theta_0|^2_{\Omega_0}\right).
\end{eqnarray*}
Hence $f_\Omega=ke^{-\varphi}$ is equivalent to:
\begin{equation}\label{eq:phi}
\Delta_{\Omega_0}\varphi+\langle d\varphi\vert\theta_0\rangle_{\Omega_0}=\frac{1}{n-1}(k-|\theta_0|^2_{\Omega_0}).
\end{equation}

Consider the linear elliptic operator $L_0$ and its $L^2$-formal adjoint $L_0^*$ with respect to $\Omega_0$, acting on $\phi\in\ce(X)$ by:
\begin{eqnarray*}
L_0\phi &=&\Delta_{\Omega_0}\phi+\langle d\phi\vert\theta_0\rangle_{\Omega_0}\\
L_0^* \phi &=&\Delta_{\Omega_0}\phi-\langle d\phi\vert\theta_0\rangle_{\Omega_0}.
\end{eqnarray*}
The theory of linear elliptic differential operators tells us that we have an $L^2$-orthogonal decomposition:
\begin{equation}\label{eq:l2dec}
\ce(X)=\ker L_0^*\oplus L_0(\ce(X)).
\end{equation}

Note that  $\ker L_0=\ker L_0^*=\RR\subset\ce(X)$. Indeed, if $\phi\in\ker L_0$ or $\phi\in\ker L_0^*$ then:
\begin{eqnarray*}
\int_X|d\phi|^2_{\Omega_0}\frac{\Omega_0^n}{n!}&=&\int_X\Delta_{\Omega_0} \phi\cdot \phi\frac{\Omega_0^n}{n!}\\
&=&\pm\int_X\langle d\phi\vert\theta_0\rangle_{\Omega_0} \phi\frac{\Omega_0^n}{n!}\\
&=&\pm\frac{1}{2}\int_X \phi^2d^{*_{\Omega_0}}\theta_0\frac{\Omega_0^n}{n!}=0
\end{eqnarray*}
hence $d\phi=0$ and $\phi$ is constant. This, together with \eqref{eq:l2dec}, implies that a function $h\in\ce(X)$ belongs to $L_0(\ce(X))$ if and only if $\int_Xh\frac{\Omega^n_0}{n!}=0$.

In particular, if we let:
\begin{equation*}
k_0:=\int_X|\theta_0|^2_{\Omega_0}\frac{\Omega^n_0}{n!}\geq 0 
\end{equation*}
then we find that $h:=\frac{1}{n-1}(k-|\theta_0|^2_{\Omega_0})\in L_0(\ce(X))$ if and only if $k=k_0$. Hence eq.~\eqref{eq:phi} admits a solution $\varphi\in\ce(X)$, unique up to addition by a constant, if and only if $k=k_0$.

Finally, it follows that the solution $\Omega$ to eq. \eqref{eq:CritFf} is unique up to multiplication by a constant, and thus is unique in $\{\Omega\}_1$. This concludes the proof.
\end{proof}

In particular, we obtain:
\begin{cor}\label{cor:minF}
Given a compact complex surface $(X,J)$ and a conformal class of normalized Hermitian metrics $\{\Omega\}_1$, there exists and is unique a critical metric $\Omega\in\{\Omega\}_1$ for the functional $\mathcal F$ restricted to $\{\Omega\}_1$. $\Omega$ is a distinguished metric and $\mathcal F (\Omega)$ is an absolute minimum for $\mathcal F$ on $\{\Omega\}_1$. 
\end{cor}
\begin{proof}
The first assertion is clear. Suppose that $\Omega$ is a distinguished metric and let $\Omega_t=\phi_t\Omega\in\{\Omega\}_1$ be a variation of $\Omega$, with:
\begin{equation*}
\phi_t=1+t\phi_1+t^2\phi_2+o(t^2).
\end{equation*}
The normalization condition implies that:
\begin{equation*}
\int_X\phi_1\frac{\Omega^2}{2}=\int_X(\phi_1^2+2\phi_2)\frac{\Omega^2}{2}=0.
\end{equation*}
We have:
\begin{eqnarray*}
dJ\theta_{\Omega_t}&=&dJ\theta_{\Omega}-\phi_t^{-2}d\phi_t\wedge d^c\phi_t+\phi_t^{-1}dd^c\phi_t\\
&=&dJ\theta_{\Omega}+tdd^c\phi_1+t^2dd^c\left(\phi_2-\frac{\phi^2_1}{2}\right)+o(t^2)
\end{eqnarray*}  
from which we deduce, using that $(dd^c)^*dJ\theta_\Omega=0$ by eq.~\eqref{extrF}:
\begin{eqnarray*}
\mathcal F(\Omega_t)-\mathcal F(\Omega) &=&\int_X\left( t^2|dd^c\phi_1|^2_\Omega+2\langle dJ\theta_\Omega\vert t dd^c\phi_1+t^2dd^c( \phi_2-\frac{\phi_1^2}{2})\rangle_\Omega \right)\frac{\Omega^2}{2}+o(t^2)\\
&=&t^2\int_X|dd^c\phi_1|_\Omega^2\frac{\Omega^2}{2}+o(t^2)
\end{eqnarray*}
which is positive for $|t|$ small enough. Thus $\mathcal F(\Omega)$ is a minimum of $\mathcal F$ on $\{\Omega\}_1$. 
\end{proof}

Note that, in general, the distinguished metrics are not Gauduchon:

\begin{cor}\label{cor:f-2-vaisman2}
Let $X$ be a compact almost complex manifold and let $\{\Omega\}_1$ be a normalized conformal class of Hermitian metrics. Then a Gauduchon metric $\Omega\in\{\Omega\}_1$ is distinguished if and only if $|\theta_\Omega|_\Omega$ is constant. In particular, any Vaisman metric (i.e. lcK metric with $\nabla^g\theta_\Omega=0$) is distinguished.
\end{cor}

\begin{proof}
By Theorem~\ref{thm:criticalF}, the Gauduchon metric $\Omega$ is distinguished if and only if $\Omega$ is balanced, so $|\theta_\Omega|_\Omega=0$, or if $|\theta_\Omega|^2_{\Omega}\Omega$ is Gauduchon. This is further equivalent to $|\theta_\Omega|_\Omega^2$ being constant, by the uniqueness of normalized Gauduchon metrics in a conformal class.
\end{proof}

\begin{cor}\label{gauduchon}
Let $X$ be a compact almost complex manifold and let $\{\Omega\}_1$ be a normalized conformal class of Hermitian metrics. If $\Omega \in \{\Omega\}_1$  has constant $|\theta_\Omega|_\Omega$, then $\Omega$  is distinguished if and only if $\Omega$ is Gauduchon.
\end{cor}

\begin{proof}
 By the proof of Theorem~\ref{thm:criticalF}, under our hypothesis $\Omega$ is distinguished if and only if $f_\Omega^{n-1}=(d^{*_\Omega}\theta_\Omega+|\theta_\Omega|^2_\Omega)^{n-1}\in\ker U_\Omega$. Furthermore, either $f_\Omega=0$, in which case $\Omega$ is Gauduchon, or $f_\Omega>0$ on $X$. In the second case,  putting $q:=d^{*_\Omega}\theta_\Omega$, this further reads:
 \begin{equation*}
 (n-1)f_\Omega^{n-3}\left(f_\Omega\Delta_\Omega q-(n-2)|dq|^2_{\Omega}\right)+(n-1)f_\Omega^{n-2}\langle dq\vert\theta_\Omega\rangle_\Omega+f_\Omega^{n-1}q=0
 \end{equation*}
 or also:
 \begin{equation*}
 (n-1)f_\Omega\Delta_\Omega q=(n-1)(n-2)|dq|^2_\Omega-(n-1)f_\Omega\langle dq\vert\theta_\Omega\rangle_\Omega-f^2_\Omega q.
 \end{equation*}

 Let $x_\mathrm{min}\in X$ be a point where $q$ attains its minimum.  Since $\int_Xq\frac{\Omega^n}{n!}=0$, if $\Omega$ is not Gauduchon then $q(x_\mathrm{min})<0$, and we find, using that $f_\Omega(x_\mathrm{min})>0$:
\begin{equation*}
0\geq (n-1)f_\Omega(x_\mathrm{min})\Delta_\Omega q(x_\mathrm{min})=-q(x_\mathrm{min})f^2(x_\mathrm{min})>0
\end{equation*}
which is a contradiction, hence $\Omega$ is Gauduchon.

Conversely, if $\Omega$ is Gauduchon then it is either balanced, or the function $f_\Omega$ is constant and positive, hence $f_\Omega\Omega$ is Gauduchon. In both cases, $\Omega$ is distinguished by Theorem~\ref{thm:criticalF}.
\end{proof}

\begin{rmk}
Let the complex surface $X^2$ be a compact quotient of a solvable Lie group endowed with an invariant complex structure. (Here, invariant means locally homogeneous, that is, it is induced by a structure on the universal cover that is invariant by left-translations.)
Compact complex surfaces of this type are tori, hyperelliptic surfaces, Inoue surfaces of type $\mathcal S_M$ and of type $\mathcal S^\pm$, and primary and secondary Kodaira surfaces \cite{hasegawa}.
Then any invariant Hermitian metric $\Omega$ is a solution of \eqref{eq:critical-F-surface}. This follows because $dd^c\Delta_\Omega\Omega$ has to be invariant, but the only invariant $4$-forms are scalar multiple of the volume form, that can not be exact by the Stokes theorem. These distinguished metrics satisfy the hypothesis of the above corollary, i.e. are Gauduchon and $|\theta_\Omega|_\Omega^2$ is constant.
\end{rmk}

\section{The functional $\mathcal A$}
We consider now the functional:
\begin{equation}\label{eq:a}
\mathcal A \colon \mathcal{H}_1 \to \mathbb R, \qquad
\mathcal A(\Omega) := \int_X | d^c \Omega |_\Omega^2 \frac{\Omega^n}{n!} = \int_X | d \Omega |_\Omega^2 \frac{\Omega^n}{n!} ,
\end{equation}
as well as its restriction to a conformal class in $\mathcal H_1$.
Notice that clearly (almost-)K\"ahler metrics are  critical points for this functional.

\begin{rmk}\label{rmk: motivation-A}
The motivation behind the study of the functional $\mathcal A$ lies in the fact that $d^c\Omega$ is the torsion of the Bismut connection, which is the unique connection $\nabla^B$ on a Hermitian manifold $(X,J,g)$ such that $\nabla^Bg=0$, $\nabla^BJ=0$ and
$c(X,Y,Z)=g(X,T^B(Y,Z))$ is totally skew-symmetric, where $T^B$ denotes the torsion of $\nabla^B$
(see {\itshape e.g.} \cite{fino-tomassini-SKT}). It turns out that $c=d^c\Omega$ and
$$
g(\nabla_X^BY,Z)=g(\nabla^g_XY,Z)+\frac{1}{2}c(X,Y,Z)\,.
$$
\end{rmk}

\begin{rmk}\label{remdOmega}
If $\{\Omega\}_1$ is an (almost) lcK conformal class, then $\mathcal A$ is a constant multiple of $\mathcal L_G$. This comes from the fact that:
\begin{eqnarray*}
(n-1)^2|d\Omega|^2_\Omega &=&|\theta_\Omega\wedge\Omega|_\Omega^2=\langle \Omega\vert\iota_{\theta_\Omega^{\sharp_\Omega}}(\theta_\Omega\wedge\Omega)\rangle_\Omega\\
&=& n|\theta_\Omega|_\Omega^2-\Lambda_\Omega(\theta_\Omega\wedge J\theta_\Omega)= (n-1)|\theta_\Omega|^2_\Omega.
\end{eqnarray*}
Hence by \cite{gauduchon}, the critical points of $\mathcal A$ restricted to $\{\Omega\}_1$ in this case are the solutions to the equation $|\theta_\Omega|^2_\Omega+2d^{*_\Omega}\theta_\Omega=k\in\RR$.
\end{rmk}

\begin{rmk} Similarly, if $X$ is a compact almost complex $4$-dimensional manifold, then $\mathcal{A}$ coincides with $\mathcal{L}_{G}$. This is because in dimension $4$ one always has $d\Omega=\theta_{\Omega} \wedge \Omega$, hence as before, $|d\Omega|^2_{\Omega}=|\theta_\Omega|^2_\Omega$.
By \cite[Th\'eor\`eme III.4]{gauduchon}, it turns out that on compact complex surfaces, the critical points of $\mathcal A$ are exactly the K\"ahler metrics, which are in fact absolute minima.

\end{rmk}

In the general case, we have the following:
\begin{prop}\label{prop:crit-points-A}
On a compact almost complex $2n$-dimensional manifold, the critical points $\Omega$ for the functional $\mathcal A$ in \eqref{eq:a} restricted to the conformal class $\{\Omega\}_1$ in $\mathcal{H}_1$ are described by the equation
\begin{equation}\label{eq:critical-A-ndim}
(n-1) | d\Omega |_{\Omega}^2 + 2 d^{*_\Omega} \theta_\Omega = k
\end{equation}
for $k$ constant.
\end{prop}

\begin{proof}
Consider a variation of $\Omega$ in its conformal class:
$$ \tilde\Omega_t = \phi_t \cdot \Omega , $$
where
$$ \phi_t=1+t\dot\phi+o(t), \quad \dot\phi\in\ce(X,\RR) \quad \text{such that} \quad \int_X \dot\phi \frac{\Omega^n}{n!}=0. $$

We compute:
\begin{eqnarray*}
\mathcal A(\tilde\Omega_t) &=&
\int_X | d^c \tilde\Omega_t |_{\tilde\Omega_t}^2 \frac{\tilde\Omega_t^n}{n!} 
= \int_X | d \tilde\Omega_t |_{\tilde\Omega_t}^2 \frac{\tilde\Omega_t^n}{n!} \\
&=& \int_X | d (\phi_t\Omega) |_{\Omega}^2 \phi_t^{n-3} \frac{\Omega^n}{n!}
= \int_X | d \phi_t \wedge \Omega + \phi_t d\Omega |_{\Omega}^2 \phi_t^{n-3} \frac{\Omega^n}{n!} \\
&=& \int_X | d \phi_t \wedge \Omega |^2 \phi_t^{n-3} \frac{\Omega^n}{n!} + \int_X | \phi_t d\Omega |_{\Omega}^2 \phi_t^{n-3} \frac{\Omega^n}{n!} + 2 \int_X \left\langle d \phi_t \wedge \Omega \middle\vert \phi_t d\Omega \right\rangle_\Omega \phi_t^{n-3} \frac{\Omega^n}{n!} \\
&=& \int_X | \phi_td\Omega |_{\Omega}^2 \phi_t^{n-3} \frac{\Omega^n}{n!} + 2 \int_X \left\langle td \dot\phi \wedge \Omega \middle\vert d\Omega \right\rangle_\Omega \frac{\Omega^n}{n!} + o(t) \\
&=& \int_X | d\Omega |_{\Omega}^2 \frac{\Omega^n}{n!} + t(n-1)\int_X | d\Omega |_{\Omega}^2 \dot\phi \frac{\Omega^n}{n!} + 2t \int_X \left\langle d \dot\phi \wedge \Omega \middle\vert d\Omega \right\rangle_\Omega \frac{\Omega^n}{n!} + o(t) \\
&=& \mathcal{A}(\Omega) + t \int_X \left((n-1) | d\Omega |_{\Omega}^2 + 2 d^{*_\Omega} \Lambda_\Omega d\Omega \right) \cdot \dot\phi \frac{\Omega^n}{n!} + o(t) \\
&=& \mathcal{A}(\Omega) + t \int_X \left((n-1) | d\Omega |_{\Omega}^2 + 2 d^{*_\Omega} \theta_\Omega \right) \cdot \dot\phi \frac{\Omega^n}{n!} + o(t)
\end{eqnarray*}
Therefore, the equation for the critical points is:
$$ (n-1) | d\Omega |_{\Omega}^2 + 2 d^{*_\Omega} \theta_\Omega = k $$
for $k$ constant.
\end{proof}

We have the following:

\begin{cor}\label{cor:g-1}
Let $X$ be a compact almost complex $2n$-dimensional manifold with $n>1$. If $\Omega\in\{\Omega\}_1$ is Gauduchon, then it is critical  for
 $\mathcal A$ on $\{\Omega\}_1$ if and only if
$| d\Omega |_{\Omega}^2$ is constant.
\end{cor}

\begin{cor}\label{cor:g-2}
Let $X$ be a compact almost complex $2n$-dimensional manifold with $n>1$. If $\Omega\in\{\Omega\}_1$ is critical for
 $\mathcal A$ and $| d\Omega |_{\Omega}^2$ is constant then
 $\Omega$ is Gauduchon.
\end{cor}

\begin{proof}
From the equation of critical points for $\mathcal{A}$ if $| d\Omega |_{\Omega}^2$ is constant then $d^{*_\Omega}\theta_{\Omega}$ is also constant. Since
$\int_Xd^{*_\Omega}\theta_{\Omega}\frac{\Omega^n}{n!}=0$ we get that $d^{*_\Omega}\theta_{\Omega}=0$.
\end{proof}

 Recall that a Hermitian metric $\Omega$ on a compact complex manifold $X$ of dimension $n$ is called \emph{strong K\"ahler with torsion} (briefly SKT, also known as pluriclosed) if $dd^c\Omega=0$. When $n=2$, SKT metrics coincide by definition with Gauduchon metrics.
When $n>2$, an SKT metric $\Omega$ is Gauduchon if and only if:
 $$
 | d\Omega |_{\Omega}^2=|\theta_\Omega|^2_\Omega=\frac{1}{n-1}|\theta_\Omega\wedge\Omega|^2_\Omega, 
 $$
see {\itshape e.g.} \cite[Equation (2.13)]{alexandrov-ivanov}.
 As a consequence, we have:

 \begin{cor}\label{cor:g-3}
 Let $X^n$ be a compact complex manifold of dimension $n>2$. If $\Omega\in\{\Omega\}_1$ is both Gauduchon and SKT, then it is critical for
 $\mathcal A$ if and only if $ | \theta_{\Omega} |_{\Omega}^2$ is constant.
 \end{cor}

\begin{rmk}
Note that Example \ref{ex:inoue-bombieri} shows that the infimum of the functional $\mathcal{A}$ on the space of Hermitian metrics of volume one can be zero even if there is no K\"ahler metric.
\end{rmk}

\section{The functional $\mathcal R$}
We consider now the functional:
\begin{equation}\label{eq:r}
\mathcal R \colon \mathcal{H}_1 \to \mathbb R, \qquad
\mathcal R(\Omega) := \int_X | dd^c \Omega |_\Omega^2 \frac{\Omega^n}{n!} ,
\end{equation}
as well as its restriction to a conformal class in $\mathcal H_1$.

The motivation behind the functional $\mathcal R$ is that SKT metrics on compact complex manifolds are by definition critical points of $\mathcal R$.

In general, we have:

\begin{prop}\label{prop:r-critical}
On a compact almost complex $2n$-dimensional manifold, the critical points $\Omega$ for the functional $\mathcal R$ in \eqref{eq:r} restricted to the conformal class $\{\Omega\}_1$ are described by the equation:
\begin{equation}\label{eq:critical-R-ndim}
(n-4) \left| dd^c\Omega\right|_{\Omega}^2 +  2\Lambda_\Omega (dd^c)^{*_\Omega} dd^c\Omega=k 
\end{equation}
for $k$ constant.
\end{prop}

\begin{proof}
Consider a variation of $\Omega$ in its conformal class:
$$ \tilde\Omega_t = \phi_t \cdot \Omega , $$
where
$$ \phi_t=1+t\dot\phi+o(t), \quad \dot\phi\in\ce(X,\RR)\quad \text{such that} \quad \int_X \dot\phi \frac{\Omega^n}{n!}=0. $$

We compute:
\begin{eqnarray*}
\mathcal R(\tilde\Omega_t) &=&
\int_X | dd^c \tilde\Omega_t |_{\tilde\Omega_t}^2 \frac{\tilde\Omega_t^n}{n!}
= \int_X | dd^c (\phi_t\Omega) |_{\Omega}^2 \phi_t^{n-4} \frac{\Omega^n}{n!} \\
&=& \int_X \left| dd^c\Omega + tdd^c(\dot\phi\Omega)\right|_{\Omega}^2 (1+t(n-4)\dot\phi) \frac{\Omega^n}{n!} + o(t) \\
&=& \mathcal{R}(\Omega) + t \int_X \left( (n-4) \left| dd^c\Omega\right|_{\Omega}^2 \cdot \dot\phi + 2\langle dd^c(\dot\phi\Omega) \vert dd^c\Omega \rangle\right) \frac{\Omega^n}{n!} + o(t) \\
&=& \mathcal{R}(\Omega) + t \int_X \left( (n-4) \left| dd^c\Omega\right|_{\Omega}^2 +  2\Lambda_\Omega (dd^c)^{*_\Omega} dd^c\Omega \right) \cdot \dot\phi  \frac{\Omega^n}{n!} + o(t). \\
\end{eqnarray*}
Therefore, the equation for the critical points is:
$$(n-4) \left| dd^c\Omega\right|_{\Omega}^2 +  2\Lambda_\Omega (dd^c)^{*_\Omega} dd^c\Omega=k $$
for $k$ constant.
\end{proof}

\begin{cor}\label{cor:critR}
Let $(X,J)$ be an almost complex compact surface and let $\{\Omega\}_1$ be a conformal class of normalized almost Hermitian metrics. Then $\Omega\in\{\Omega\}_1$ is critical for $\mathcal F$ if and only if $\Omega$ is Gauduchon.
\end{cor}
\begin{proof}
Let us first note that in the case $n=2$, eq.~\eqref{eq:critical-R-ndim} is equivalent to:
\begin{equation}\label{eq:critR2}
|dd^c\Omega|_\Omega^2=\Lambda_\Omega (dd^c)^{*_\Omega}dd^c\Omega.
\end{equation}
Indeed, this follows by integrating eq.~\eqref{eq:critical-R-ndim} over $X$:
\begin{equation*}
k\cdot\mathrm{vol}_\Omega(X)=-2\mathcal{R}(\Omega)+2\int_X\langle \Omega\vert (dd^c)^{*_\Omega}dd^c\Omega\rangle_\Omega\frac{\Omega^2}{2}=0. 
\end{equation*}

Furthermore, note that since $n=2$, we have $dd^c\Omega=-d^{*_\Omega}\theta_\Omega\cdot\frac{\Omega^2}{2}$. Denoting by $q:=-d^{*_\Omega}\theta_\Omega$, we find:
\begin{eqnarray*}
\Lambda_\Omega(dd^c)^*dd^c\Omega &=&\langle \Omega\vert-*_\Omega dd^c*_\Omega(q\frac{\Omega^2}{2})\rangle_\Omega\\
&=&\langle *_\Omega\Omega\vert -dd^cq\rangle_\Omega\\
&=&-\Lambda_{\Omega}(dd^cq)\\
&=&\Delta_\Omega q+\langle dq\vert\theta_\Omega\rangle_\Omega
\end{eqnarray*}
where for the last equality, we used eq.~\eqref{eq:ddc}. Thus, eq.~\eqref{eq:critR2} is equivalent to:
\begin{equation}\label{eq:crit-R2}
q^2=\Delta_\Omega q+\langle dq\vert\theta_\Omega\rangle_\Omega.
\end{equation}

Now clearly Gauduchon metrics satisfy eq.~\eqref{eq:crit-R2} and minimize $\mathcal R$. Conversely, suppose that $\Omega$ satisfies eq.~\eqref{eq:crit-R2}, and suppose by contradiction that $\Omega$ is not Gauduchon. Then at a minimum point $x_{\mathrm{min}}\in X$ of $q$ we have $q(x_{\mathrm{min}})<0$, from which it follows:
\begin{equation*}
0<q^2(x_{\mathrm{min}})=\Delta_\Omega q(x_{\mathrm{min}})\leq 0
\end{equation*}
which is impossible. Hence $\Omega$ is Gauduchon.     
\end{proof}

\section{An example: the Inoue-Bombieri surface}\label{ex:inoue-bombieri}
We consider an Inoue-Bombieri surface of type $S_M$ \cite{inoue, bombieri}. As described in \cite{hasegawa}, it can be viewed as a compact quotient of a solvable Lie group. Consider the complex structure described by the coframe of invariant $(1,0)$-forms $\{\varphi^1,\varphi^2\}$ with structure equations:
$$ d\varphi^1=\frac{1}{2\sqrt{-1}}\varphi^1\wedge\varphi^2-\frac{1}{2\sqrt{-1}}\varphi^1\wedge\bar\varphi^2, \qquad d\varphi^2=-\sqrt{-1}\varphi^2\wedge\bar\varphi^2 . $$
Any invariant Hermitian metric has associated $(1,1)$-form:
\begin{equation}\label{eq:metric}
\Omega = \sqrt{-1}r^2 \varphi^1\wedge\bar\varphi^1 + \sqrt{-1}s^2 \varphi^2\wedge\bar\varphi^2 + u \varphi^1\wedge\bar\varphi^2 - \bar u \varphi^2\wedge\bar\varphi^1
\end{equation}
where the parameters $r,s\in\mathbb R$ and $u\in\mathbb C$ satisfy $r^2>0$, $s^2>0$, $r^2s^2 - |u|^2>0$. 

Let $c:=\int_{M} -\varphi^1 \wedge \bar\varphi^1 \wedge \varphi^2 \wedge \bar \varphi^2$. Then any choice of $r, s, u$ with $r^2s^2 -|u|^2=\frac{1}{c}$ gives a metric in $\mathcal{H}_1$.

We compute:
$$ d\Omega = r^{2}  \varphi^1 \wedge \varphi^{2} \wedge \bar\varphi^{1} + \frac{\sqrt{-1}}{2} u  \varphi^{1} \wedge \varphi^{2} \wedge \bar\varphi^{2} + r^{2}  \varphi^{1} \wedge \bar\varphi^{1} \wedge \bar\varphi^{2} - \frac{\sqrt{-1}}{2} \bar u  \varphi^{2} \wedge \bar\varphi^{1} \wedge \bar\varphi^{2} .$$
We get that the corresponding Lee form is:
$$ \theta = \frac{3 \, r^{2} u}{2 \, {\left(r^{2} s^{2} - |u|^2 \right)}} \varphi^1 + \sqrt{-1} \, \frac{2 \, r^{2} s^{2} + |u|^2 }{2 \, {\left(r^{2} s^{2} - |u|^2 \right)}} \varphi^2 + \frac{3 \, r^{2} \overline{u}}{2 \, {\left(r^{2} s^{2} - |u|^2 \right)}} \bar\varphi^1 - \sqrt{-1} \, \frac{2 r^{2} s^{2} + |u|^2 }{2 \, {\left(r^{2} s^{2} - |u|^2 \right)}} \bar\varphi^2 . $$

\begin{itemize}
\item Of course, since $d^{*}\theta$ is an invariant function, then $d^*\theta=0$, whence:
$$ \mathcal{G}(\Omega)=0 $$
for any invariant metric $\Omega$.

\item We compute:
\begin{eqnarray*}
dJ\theta &=&
\left( \frac{3 \, r^{2} u}{4 \, {\left(r^{2} s^{2} - |u|^2 \right)}} \right)  \varphi^{1} \wedge \varphi^{2} + \left( \frac{3 \, r^{2} \overline{u}}{4 \, {\left(r^{2} s^{2} - |u|^2\right)}} \right)  \bar\varphi^{1} \wedge \bar\varphi^{2} \\
&& + \left( -\frac{3 \, r^{2} u}{4 \, {\left(r^{2} s^{2} - |u|^2\right)}} \right)  \varphi^{1} \wedge \bar\varphi^{2} + \left( \frac{3 \, r^{2} \overline{u}}{4 \, {\left(r^{2} s^{2} - |u|^2\right)}} \right)  \varphi^{2} \wedge \bar\varphi^{1} \\
&& + \left( \frac{2 \sqrt{-1} \, r^{2} s^{2} + \sqrt{-1} \, |u|^2}{2 \, {\left(r^{2} s^{2} - |u|^2\right)}} - \frac{-2 \sqrt{-1} \, r^{2} s^{2} - \sqrt{-1} \, |u|^2}{2 \, {\left(r^{2} s^{2} - |u|^2\right)}} \right)  \varphi^{2} \wedge \bar\varphi^{2} .
\end{eqnarray*}
Thus, we infer:
\begin{eqnarray*}
\frac{dJ\theta \wedge *\overline{dJ\theta}}{\varphi^1\wedge\varphi^2\wedge\bar\varphi^1\wedge\bar\varphi^2} &=&
\frac{9 \, r^{4} |u|^2}{8 \, {\left(r^{2} s^{2} - |u|^2\right)}^{2}} + \frac{3 \, {\left(11 \, r^{4} s^{2} \overline{u} + 7 \, r^{2} |u|^{4}\right)} r^{2} u}{16 \, {\left(r^{2} s^{2} - |u|^2\right)^3}} \\
&& + \frac{3 \, {\left(11 \, r^{4} s^{2} + 7 \, r^{2} |u|^{2}\right)} r^{2} |u|^2}{16 \, {\left(r^{2} s^{2} - |u|^2\right)^3}} \\
&& + \frac{{2\left(-4 \sqrt{-1} \, r^{6} s^{2} - 5 \sqrt{-1} \, r^{4} |u|^2\right)} {\left({2 \sqrt{-1} \, r^{2} s^{2} + \sqrt{-1} \, |u|^2}\right)}}{4 \, {\left(r^{2} s^{2} - |u|^{2} \right)^3}}.
\end{eqnarray*}

For example, when $u=0$, we find:
\begin{equation*}
dJ\theta = 2\sqrt{-1} \, \varphi^{2} \wedge \bar\varphi^{2} , \ \   \left| dJ\theta \right|^2 = 4\frac{1}{s^4}.
\end{equation*}

Thus, making $s$ go to $\infty$ and taking $r^2=\frac{1}{cs^2}$, we conclude that the infimum for the functional $\mathcal F$ on $\mathcal{H}_1$ is $0$. However, $0$ is not a minimum, since $dJ\theta_\Omega=0$ would give $dd^{*_\Omega}\Omega=0$ whence $|d^{*_\Omega}\Omega|^2_\Omega=|d\Omega|^2_\Omega=0$, but $S_M$ does not admit any K\"ahler metric.

\item We compute:
\begin{eqnarray*}
|d^c\Omega|_\Omega^2 \frac{\Omega^2}{2} &=& |d\Omega|^2_\Omega \frac{\Omega^2}{2} = d\Omega\wedge*_\Omega d\Omega \\
&=& \left( \frac{3 \, r^{2} |u|^2 }{2 \, {\left(r^{2} s^{2} - |u|^2 \right)}} + \frac{2{\left(2 \, r^{2} s^{2} + |u|^2 \right)} r^{2}}{2(r^{2} s^{2} - |u|^2) } \right)  \varphi^{12\bar1\bar2} \\
&=& \left( \frac{5 \, r^{2} |u|^2 +4 \, r^{4} s^{2} }{2 \, {\left(r^{2} s^{2} - |u|^2 \right)}} \right)  \varphi^{12\bar1\bar2} ,
\end{eqnarray*}
which for example for $u=0$ gives:
$$ |d^c\Omega|_\Omega^2 \frac{\Omega^2}{2} = 2r^2 \varphi^{12\bar1\bar2}. $$
Thus, making $r$ go to $\infty$ and taking $s^2=\frac{1}{cr^2}$, we conclude again that the infimum of $\mathcal A$ on $\mathcal{H}_1$ is $0$ and is not attained.

\item As for the last functional, any invariant metric is Gauduchon by the Stokes theorem, whence:
$$ dd^c\Omega=0 . $$
\end{itemize}

\bigskip

{\small
\noindent{\sl Acknowledgments.}
The authors are very grateful to Paul Gauduchon for his teachings, support, and suggestions.
We also thank Andrei Moroianu for useful remarks.
Many thanks also to Max Pontecorvo, whom we wish a belated happy birthday!\\
Most of the work has been written during the fourth-named Author's post-doctoral fellow at the Dipartimento di Matematica e Informatica ``Ulisse Dini'' of the Universit\`{a} di Firenze. She is grateful to the department for the hospitality.
}


\begin{thebibliography}{99}

\bibitem[AI01]{alexandrov-ivanov}
B. Alexandrov, S. Ivanov, Vanishing theorems on Hermitian manifolds, {\em Differential Geom. Appl.} \textbf{14} (2001), no. 3, 251--265.

\bibitem[Bes87]{besse} A. Besse, {\em Einstein Manifolds}, Springer Verlag, 1987

\bibitem[Bom73]{bombieri}
E. Bombieri, Letter to Kodaira, 1973.

\bibitem[FT09]{fino-tomassini-SKT}
A. Fino, A. Tomassini, A survey on strong KT structures, {\em Bull. Math. Soc. Sci. Math. Roumanie (N.S.)} \textbf{52(100)} (2009), no. 2, 99--116.

\bibitem[Gau77a]{gauduchon-CRAS}
P. Gauduchon, Le th\'eor\`eme de l'excentricit\'e nulle, {\em C. R. Acad. Sci. Paris S\'er. A-B} \textbf{285} (1977), no. 5, A387--A390.



\bibitem[Gau84]{gauduchon}
P. Gauduchon, La $1$-forme de torsion d'une vari\'et\'e hermitienne compacte, {\em Math. Ann.} \textbf{267} (1984), no. 4, 495--518.

\bibitem[Gau97]{gauduchon-bumi}
P. Gauduchon, Hermitian connections and Dirac operators, {\em Boll. Un. Mat. Ital. B (7)} \textbf{11} (1997), no. 2, suppl., 257--288.



\bibitem[GGP08]{grantcharov-grantcharov-poon}
D. Grantcharov, G. Grantcharov, Y. S. Poon, Calabi-Yau connections with torsion on toric bundles, {\em J. Differential Geom.} \textbf{78} (2008), no. 1, 13--32.

\bibitem[Has05]{hasegawa}
K. Hasegawa, Complex and K\"ahler structures on compact solvmanifolds, Conference on Symplectic Topology, {\em J. Symplectic Geom.} \textbf{3} (2005), no. 4, 749--767.

\bibitem[Hop27]{hopf} E. Hopf, Elementare Bemerkungen \"uber die L\"osungen partieller Differentialgleichungen zweiter Ordnung vom elliptischen Typus. Sitzungber. Preuss. Akad. Wiss. Phys. Math. Kl. \textbf{19}, (1927) 147--152

\bibitem[Lie54]{libermann}
P. Libermann, Sur les connexions hermitiennes, {\em C. R. Acad. Sci. Paris} \textbf{239} (1954), 1579--1581.

\bibitem[Ino74]{inoue}
M. Inoue, On surfaces of Class $VII_{0}$, {\em Invent. Math.} \textbf{24} (1974), 269--310.


\bibitem[Vai90]{vaisman}
I. Vaisman, On some variational problems for $2$-dimensional Hermitian metrics, {\em Ann. Global Anal. Geom.} \textbf{8} (1990), no. 2, 137--145. 

\end{thebibliography}
\end{document}